\documentclass[a4paper,11pt]{amsart}

\usepackage{amsfonts,latexsym,rawfonts,amsmath,amssymb,amsthm, mathrsfs, lscape}
\usepackage[all]{xy}
\usepackage[english]{babel}
\usepackage[utf8]{inputenc}
\usepackage{xfrac}
\usepackage{multirow}
\usepackage{comment}

\usepackage{array, tabularx}

\usepackage{setspace}
\usepackage{textcomp}

\usepackage{spverbatim}

\usepackage{amsthm}
\usepackage{tikz}
\usetikzlibrary{cd}

\usepackage[hypertexnames=false,
backref=page,
    pdftex,
    pdfpagemode=UseNone,
    breaklinks=true,
    extension=pdf,
    colorlinks=true,
    linkcolor=blue,
    citecolor=blue,
    urlcolor=blue,
]{hyperref}

\usepackage[top=1.5in, bottom=1in, left=0.9in, right=0.9in]{geometry}

\renewcommand{\Re}{\mathsf{Re}\,}
\renewcommand{\Im}{\mathsf{Im}\,}

\allowdisplaybreaks[1]

\newtheorem{theorem}{Theorem}
\newtheorem{corollary}[theorem]{Corollary}
\newtheorem{lemma}[theorem]{Lemma}

\newtheorem{proposition}[theorem]{Proposition}
\newtheorem{remark}[theorem]{Remark}

\newcommand{\referenza}{}
\newtheorem*{theorem*}{Theorem \referenza}
\newtheorem*{corollary*}{Corollary \referenza}
\newtheorem*{proposition*}{Proposition \referenza}

\makeatletter
\@namedef{subjclassname@2020}{\textup{2020} Mathematics Subject Classification}
\makeatother

\title{On metric and cohomological properties of Oeljeklaus-Toma manifolds}

\author{Daniele Angella}
\author{Art\={u}ras Dubickas}
\author{Alexandra Otiman}
\author{Jonas Stelzig}

\address[Daniele Angella]{Dipartimento di Matematica e Informatica ``Ulisse Dini''\\
Universit\`a degli Studi di Firenze,
viale Morgagni 67/a, IT-50134 Firenze, Italy}
\email{daniele.angella@unifi.it}
\email{daniele.angella@gmail.com}

\address[Art\={u}ras Dubickas]{Institute of Mathematics\\
Department of Mathematics and Informatics\\
Vilnius University\\
Naugarduko 24, LT-03225 Vilnius, Lithuania}
\email{arturas.dubickas@mif.vu.lt}

\address[Alexandra Otiman]{Dipartimento di Matematica e Informatica ``Ulisse Dini''\\
Universit\`a degli Studi di Firenze,
viale Morgagni 67/a, IT-50134 Firenze, Italy and Institute of Mathematics “Simion Stoilow” of the Romanian Academy, Calea Grivitei Street 21, 010702, Bucharest, Romania
}
\email{alexandraiulia.otiman@unifi.it}
\email{alexandra.otiman@gmail.com}

\address[Jonas Stelzig]{Mathematisches Institut der Ludwig-Maximilians-Universit\"at M\"unchen,
Theresienstraße 39, DE-80993 M\"unchen}
\email{jonas.stelzig@math.lmu.de}

\thanks{
The first-named author is supported by project PRIN2017 ``Real and Complex Manifolds: Topology, Geometry and holomorphic dynamics'' (code 2017JZ2SW5), and by GNSAGA of INdAM. The third-named author is supported by GNSAGA of INdAM and by a grant of Ministry of Research and Innovation, CNCS - UEFISCDI, project no. PN-III-P4-ID-PCE-2020-0025, within PNCDI III.
}

\newcommand{\Id}{\operatorname{Id}}
\newcommand{\del}{\partial}
\newcommand{\delbar}{\bar\partial}
\newcommand{\mult}{\operatorname{mult}}
\newcommand{\gr}{\operatorname{gr}}

\newcommand{\img}[2][1]{\begin{gathered}\includegraphics[scale=#1]{#2}\end{gathered}}

\begin{document}

\begin{abstract}
We study metric and cohomological properties of Oeljeklaus-Toma manifolds. In particular, we describe the structure of the double complex of differential forms and its Bott-Chern cohomology and we characterize the existence of pluriclosed (aka SKT) metrics in number-theoretic and cohomological terms. Moreover, we prove they do not admit any Hermitian metric $\omega$ such that $\partial \overline{\partial} \omega^k=0$, for $2 \leq k \leq n-2$ and we give explicit formulas for the Dolbeault cohomology of Oeljeklaus-Toma manifolds admitting pluriclosed metrics.
\end{abstract}

\maketitle

\section*{Introduction}
The class of Oeljeklaus-Toma manifolds \cite{OT} consists of complex non-K\"ahler manifolds including Inoue-Bombieri surfaces \cite{bombieri, inoue, tricerri} of type $S_M$. They were constructed by number field techniques in \cite{OT}, where they are presented as counterexamples to a conjecture by Vaisman concerning locally conformally K\"ahler metrics.
Because of their very construction, some of their geometric properties are encoded in the algebraic structure and can be investigated by number theoretic techniques, see {\itshape e.g.} \cite{OT, v, dub, dv, bra, ot21, io19, o20, apv16, tomassini-torelli, sverbitsky-curves, sverbitsky-surfaces, k13, fkv15}.

\medskip

In this paper, we proceed with the study of metric and cohomological properties of Oeljeklaus-Toma manifolds, after \cite{OT, dub, fkv15, o20, io19, k20, ot21}.
In what follows, $X(K, U)= \mathbb{H}^s\times\mathbb{C}^t \slash \mathcal{O}_K \rtimes U$ will denote an Oeljeklaus-Toma manifold of complex dimension $n=s+t$, associated with a number field $K\simeq\mathbb Q[X]/(f)$, where $f$ has $s\geq1$ real roots and $2t\geq2$ complex roots, and to the admissible subgroup $U$ of rank $s$ of totally positive units, and $\mathcal{O}_K$ denotes the ring of rank $s+2t$ of algebraic integers, see Section \ref{subsec:ot} for details on the construction.

\medskip

We study the existence of ``special'' Hermitian metrics, that is, Hermitian metrics whose associated $(1,1)$-form solve some PDE equation. Some results are already known by \cite{OT, dub, fkv15, o20}. Here, we focus in particular on {\em pluriclosed metrics} (aka SKT, Strong K\"ahler with Torsion), namely, Hermitian metrics such that the associated $(1,1)$-form satisfies $\partial\overline\partial\omega=0$. They play a fundamental role in the Hermitian geometry of complex manifolds with applications to String Theory, see {\itshape e.g.} \cite{bismut, fino-tomassini, ggp, streets-tian, cavalcanti, ivanov-papadopoulos}.
By \cite{o20}, it is known that the existence of pluriclosed metrics is equivalent to the number theoretic properties that $s\leq t$ and, for any $u\in U$, it holds $\sigma_j(u)|\sigma_{s+j}(u)|^2=1$ for any $j\in\{1,\ldots,s\}$, and $|\sigma_{s+j}(u)|^2=1$ for any $j\in\{s+1,\ldots,t\}$, where $\sigma_1, \ldots, \sigma_s\colon K \to \mathbb R$ denote the real embeddings and $\sigma_{s+1}, \ldots,\sigma_{s+t}, \sigma_{s+t+1}=\bar\sigma_{s+1}, \ldots, \sigma_{s+2t}=\bar \sigma_{s+t} \colon K \to \mathbb C$ denote the complex embeddings induced by the roots of $f$.
In Theorem \ref{unitthm}, by number theoretic arguments, we prove that these conditions cannot hold unless $s=t$. We then get
\renewcommand{\referenza}{\ref{rem:dubickas}}
\begin{corollary*}
Let $X(K, U)$ be an Oeljeklaus-Toma manifold of type $(s, t)$.
It admits pluriclosed metrics if and only if $s = t$ and, for any $u\in U$,
$$\sigma_j(u)|\sigma_{s+j}(u)|^2=1 , \quad \text{ for any } j\in\{1,\ldots,s\} .$$
\end{corollary*}
While the above statement can be read as a number theoretic characterization of the existence of pluriclosed metrics, we will provide also a cohomological characterization in Theorem \ref{thm:cohom-skt}.

We also study {\em astheno-K\"ahler metrics}, namely, Hermitian metrics whose associated $(1,1)$-form satisfies $\partial\overline\partial\omega^{n-2}=0$. They appeared in \cite{jost-yau}, see {\itshape e.g.} \cite{fino-tomassini-2} and the references therein. More generally, Hermitian metrics with $\partial\overline\partial \omega^k=0$, for $k\in\{1,\ldots, n-2\}$, see also \cite{fu-wang-wu}, play a role in pluripotential theory on compact Hermitian manifolds, see \cite{dinew}.
In Corollary \ref{cor:ak}, we show that Oeljeklaus-Toma manifolds never admit astheno-K\"ahler metrics, as a consequence of a more general result on the inexistence of Hermitian metrics such that $\partial\overline\partial\omega^k=0$, see Proposition \ref{prop:ddcomegak}.
In Proposition \ref{prop:strongly-gauduchon}, we also show that Oeljeklaus-Toma manifolds never admit {\em strongly Gauduchon} metrics in the sense of \cite{pop13}.

In Table \ref{table:metric}, we summarize what is known until now concerning special metrics on Oeljeklaus-Toma manifolds. (The conditions are intended after possibly relabeling the embeddings.)

\begin{center}
\begin{table}[h!]
\renewcommand{\arraystretch}{1.4}
{\resizebox{\textwidth}{!}{
\begin{tabular}{cc|c|c|}
\multicolumn{2}{c|}{\bfseries metric} & {\bfseries conditions} & {\bfseries ref.}\\
\multicolumn{2}{c|}{\small ($\omega$ positive $(1,1)$-form; $\Omega$ non-degenerate $2$-form)} && \\
\hline\hline
K\"ahler & $d\omega=0$ & never & \cite{OT} \\
\hline
taming symplectic & $d\Omega=0$, $\Omega(\_,J\_)>0$ & never & \cite{fkv15} \\
\hline
balanced & $d\omega^{n-1}=0$ & never & \cite{o20} \\
\hline
\multirow{2}{*}{pluriclosed} & \multirow{2}{*}{$\partial\overline\partial\omega=0$} & $s=t$ and $\sigma_j(u)|\sigma_{s+j}(u)|^2=1$ & \multirow{2}{*}{\cite{o20}, Cor. \ref{rem:dubickas}} \\
&&($\forall u \in U, \forall j \in\{1,\ldots,s\}$) &\\
\hline
``special'' $k$-Gauduchon & \multirow{2}{*}{$\partial\overline\partial\omega^k=0$} & \multirow{2}{*}{never} & \multirow{2}{*}{Prop. \ref{prop:ddcomegak}} \\
($2\leq k \leq n-2$, $n\geq4$) &&&\\
\hline
astheno-K\"ahler & \multirow{2}{*}{$\partial\overline\partial\omega^{n-2}=0$} & \multirow{2}{*}{never} & \multirow{2}{*}{Cor. \ref{cor:ak}} \\
($n\geq4$) &&& \\
\hline
Gauduchon & $\partial\overline\partial\omega^{n-1}=0$ & always & \cite{gauduchon-cras1977} \\
\hline
strongly Gauduchon & $\partial\omega^{n-1}$ is $\overline\partial$-exact & never & Prop. \ref{prop:strongly-gauduchon} \\
\hline
\multirow{3}{*}{LCK} & \multirow{3}{*}{$d\omega=\theta\wedge\omega$, $d\theta=0$} & $|\sigma_{s+1}(u)|=\cdots=|\sigma_{s+t}(u)|$ \quad ($\forall u \in U$) & \multirow{3}{*}{\cite{OT, v, dub, dv}} \\
\cline{3-3}
&& when $t=1$: always &\\
&& when $t\geq2$: never &\\
\hline
LCB & $d\omega^{n-1}=\theta\wedge\omega^{n-1}$, $d\theta=0$ & always & \cite{o20}\\
\hline\hline
\end{tabular}
}}\medskip
\caption{Metric properties of Oeljeklaus-Toma manifolds.}
\label{table:metric}
\end{table}
\end{center}

The cohomological properties of Oeljeklaus-Toma manifolds are studied in \cite{OT, tomassini-torelli, io19, k20, ot21}. In particular, it is known that the first Betti number equals the number $s$ of real places \cite{OT} and that the second Betti number equals ${s \choose 2}$ \cite{OT, apv16}. More generally, the Betti numbers \cite{io19} and the Hodge numbers \cite{ot21} can be computed in terms of the non-trivial relations between the representations of $U$ induced by the embeddings, see Theorem \ref{thm:ot-cohom} for a precise statement.
Here, we completely describe the {\em $E_1$-isomorphism type} \cite{s21} of the double complex of differential forms on Oeljeklaus-Toma manifolds, which in turn allows one to compute the dimensions of any cohomological invariant, including the Bott-Chern cohomology, the Aeppli cohomology, the Varouchas groups, any higher-page analogues, and the groups appearing in the Schweitzer complex.

More precisely, we recall that any bounded double complex $A$ over a field $K$ (in particular, the double complex $A_X$ of complex differential forms on a complex manifold $X$) can be decomposed as a direct sum of indecomposable double complexes that are of type {\em square}, {\itshape i.e.}
$$
\begin{tikzcd}
	K \ar[r, "\simeq"] & K\\
	K \ar[u, "\simeq"] \ar[r, "\simeq"] & K \ar[u, "\simeq"]
\end{tikzcd}
$$
or type {\em zigzag} (with length equal to the number of non-zero components), {\itshape i.e.}
$$
\begin{tikzcd}
	K
\end{tikzcd},
\qquad
\begin{tikzcd}
	K \ar[r, "\simeq"] & K
\end{tikzcd},
\qquad
\begin{tikzcd}
	K \\ K \ar[u, "\simeq"]
\end{tikzcd},
\qquad
\begin{tikzcd}
	K & \\
	K \ar[u, "\simeq"] \ar[r, "\simeq"] & K
\end{tikzcd},
\qquad
\begin{tikzcd}
	K \ar[r, "\simeq"] & K \\
	& K \ar[u, "\simeq"]
\end{tikzcd},
\qquad \dots .
$$
See \cite[Theorem A]{s21} and also \cite{kq20}.
One notices that squares do not contribute to any cohomology, even length zigzags correspond to differentials in the Fr\"olicher spectral sequences, while odd length zigzags correspond to de Rham classes, and the dimension of Bott-Chern and Aeppli cohomologies equals the number of top-right, respectively bottom-left, corners in the zigzags, see \cite{s21}. In particular, Betti numbers, Hodge numbers, Bott-Chern and Aeppli numbers are linear combinations of the multiplicities $\mathrm{mult}_{Z}(A)$ for zigzags $Z$, see \cite[Corollary B]{s21}.
We also recall that an {\em $E_1$-isomorphism} is a morphism between bounded double complexes inducing isomorphisms between the column (`Dolbeault') cohomologies of source and target and between the row (`conjugate Dolbeault') cohomologies of source and target, see \cite[Definition D]{s21}. By \cite[Proposition E]{s21}, the $E_1$-isomorphism type of a bounded double complex is completely described by the multiplicities of all non-projective indecomposable bicomplexes ({\itshape i.e.} all `zigzags').

For an Oeljeklaus-Toma manifold $X$, we will define a graded subalgebra $VB = \bigoplus_r V^rB \subset A_X$ of the algebra of complex differential forms, see Equations \eqref{eq:V} and \eqref{eq:VB}, such that the inclusion is an $E_1$-isomorphism, see Lemma \ref{lem: inclusion VB-AX quiso}. Motivated by the above discussion, and thanks to \cite{ot21, k20}, we obtain the following description of the double complex of an Oeljeklaus-Toma manifold:

\renewcommand{\referenza}{\ref{thm: E1-isotype OT mfds}}
\begin{theorem*}
Let $X$ be an Oeljeklaus-Toma manifold.
For any even length zigzag $Z$, one has
\[
\mult_{[Z]}(A_X)=0.
\]
For odd length zigzags of shape $S_d^{p,q}$, one has
\[
\operatorname{mult}_{S_d^{p,q}}(A_X)=
\begin{cases}
0 & \text{ if } p\neq q \\
\operatorname{mult}_{S_d^{p,p}}(V^pB) = h^{p,d-p}_{\overline\partial}(V^pB)=\sum_{q_1+q_2=d-p}{s\choose q_1}\rho_{p,q_2} & \text{ if } p=q .
\end{cases}
\]
\end{theorem*}

As explained above, Theorem \ref{thm: E1-isotype OT mfds} allows one to compute any cohomological invariant of $X$, by only computing how this invariant evaluates on odd length zigzags. In particular, we compute the Bott-Chern cohomology:

\renewcommand{\referenza}{\ref{cor:bc-ot}}
\begin{corollary*}
Let $X(K, U)$ be an Oeljeklaus-Toma manifold of type $(s, t)$.
The Bott-Chern numbers of $X$ are given by:
\[
h_{BC}^{p,q}(X)=\sum_{r} h_{BC}^{p,q}(V^rB), \quad
\text{ where } \>\>
	h_{BC}^{p,q}(V^rB)=
\begin{cases}
	h_{\bar\partial}^{r, p+q-r } &\text{ if } r\geq p,r\geq q,\\
	h_{\bar\partial}^{r, p+q-r-1} &\text{ if } r<p,r<q,\\
	0&\text{ otherwise.}
\end{cases}
\]
\end{corollary*}

Finally, in Section \ref{subsec:cohom-skt}, we focus on Oeljeklaus-Toma manifolds admitting pluriclosed metrics and study their cohomological properties. In particular, we prove that the Dolbeault cohomology is invariant with respect to the solvmanifold structure of Oeljeklaus-Toma manifolds (see Remark \ref{rem:inv-pluriclosed} for more details) and we explicitly compute the Betti and Hodge numbers:

\renewcommand{\referenza}{\ref{cohomology}}
\begin{theorem*}
Let $X(K, U)$ be an Oeljeklaus-Toma manifold admitting a pluriclosed metric. Then the Betti and Hodge numbers are:
$$b_\ell=\sum_{k \leq s} {s \choose \ell-3k} \cdot {s \choose k} ,
\quad\quad
h_{\overline\partial}^{p, q}=  \left\{
    \begin{array}{ll}
      {s \choose q-\frac{p}{2}} \cdot {s \choose \frac{p}{2}},  & \mbox{if } p \, \,  \mbox{even}\\
        0, & \mbox{otherwise.}
    \end{array} \right. $$
\end{theorem*}

We also provide the following cohomological characterization of the pluriclosed condition:
\renewcommand{\referenza}{\ref{thm:cohom-skt}}
\begin{theorem*}
Let $X(K, U)$ be an Oeljeklaus-Toma manifold of type $(s,t)$. The following statements are equivalent:
\begin{enumerate}
    \item $X(K, U)$ admits a pluriclosed metric;
    \item $s=t$, $h_{\overline\partial}^{2, 1}=s$, $h_{\overline\partial}^{4, 2}={s \choose 2}$, $h_{\overline\partial}^{1, 2}=0$ and the natural map
    \begin{equation*}
       \bigwedge: H_{\overline{\partial}}^{2, 1}(X) \times H_{\overline{\partial}}^{2, 1}(X) \rightarrow H_{\overline{\partial}}^{4, 2}(X)
    \end{equation*}
    is surjective.
    \end{enumerate}
\end{theorem*}

\bigskip

The paper is organized as follows.
In Section \ref{sec:ot}, we recall the construction of Oeljeklaus-Toma manifolds and their solvmanifold structure.
In Section \ref{sec:metric}, we investigate metric properties of Oeljeklaus-Toma manifolds, in particular, the existence of pluriclosed metrics (Theorem \ref{unitthm} and Corollary \ref{rem:dubickas}) and astheno-K\"ahler metrics (Proposition \ref{prop:ddcomegak} and Corollary \ref{cor:ak}).
In Section \ref{sec:cohom}, we study the double complex of an Oeljeklaus-Toma manifold (Theorem \ref{thm: E1-isotype OT mfds}) and its Bott-Chern cohomology (Corollary \ref{cor:bc-ot}), and then focus on the cohomologies of pluriclosed Oeljeklaus-Toma manifolds (Theorem \ref{cohomology} and Theorem \ref{thm:cohom-skt}).

\bigskip

{
\noindent{\small
{\itshape Acknowledgments.} The authors thank Matei Toma and Victor Vuletescu for many interesting conversations. The last named author further expresses his gratitude to the first and third named authors and the University of Florence for the hospitality and excellent working conditions during a stay in the course of which a part of this article was written.
Many thanks also to the anonymous Referees for their careful reading and the useful comments, that improved the presentation of the paper.
}
}

\section{The class of Oeljeklaus-Toma manifolds}\label{sec:ot}

\subsection{Construction}\label{subsec:ot}
We briefly recall their construction.
Fix $s,t\in\mathbb N\setminus\{0\}$.
Let $K$ be an algebraic number field, say,
$$K\simeq\mathbb{Q}[X]\slash(f),$$
where $f\in\mathbb{Q}[X]$ is a monic irreducible polynomial of degree $n=[K:\mathbb{Q}]$ with $s$ real roots, and $2t$ complex roots.
Such a field always exists \cite[Remark 1.1]{OT}.
Denote by $a_1,\ldots,a_s$ the real roots of $f$, and by $a_{s+1},\ldots,a_{s+t},a_{s+t+1}=\bar a_{s+1},\ldots,a_{s+2t}=\bar a_{s+t}$ the complex non-real roots of $f$.
By considering $X\mapsto a_j$, he field $K$ admits $s+2t$ embeddings in $\mathbb{C}$ by the conjugate roots, more precisely, $s$ real embeddings and $2t$ complex embeddings:
$$
\begin{gathered}
\sigma_1,\ldots,\sigma_s\colon K\to \mathbb{R},\\
\sigma_{s+1},\ldots,\sigma_{s+t},\sigma_{s+t+1}=\overline\sigma_{s+1},\ldots,\sigma_{s+2t}=\overline\sigma_{s+t} \colon K\to\mathbb{C}.
\end{gathered}
$$

The ring $\mathcal O_K$ of algebraic integers of $K$ is a finitely-generated free Abelian group of rank $n$.
The multiplicative group $\mathcal O_K^*$ of units of $\mathcal O_K$ is a finitely-generated free Abelian group of rank $s+t-1$, by the Dirichlet unit theorem.
The group $\mathcal O_K^{*,+}$ of totally positive units is a finite index subgroup of $\mathcal O_K^*$.

Denote by $\mathbb H:=\{z\in\mathbb C : \Im z>0\}$ the upper half-plane.
By \cite[page 162]{OT}, one can always choose an {\em admissible subgroup} $U\subset\mathcal{O}_K^{*,+}$, of rank $s$, such that the fixed-point-free action
$$ \mathcal{O}_K\rtimes U\circlearrowleft\mathbb H^s\times\mathbb C^t $$
induced by
\begin{equation*}
\begin{gathered}
\begin{split}
T\colon &\mathcal{O}_K \circlearrowleft \mathbb H^s\times\mathbb C^t, \\
T_a(w_1,\ldots,w_s,z_{s+1},\ldots,z_{s+t}) &:= (w_1+\sigma_1(a),\ldots,z_{s+t}+\sigma_{s+t}(a)),
\end{split}
\\
\begin{split}
R\colon &\mathcal{O}_K^{*,+} \circlearrowleft \mathbb H^s\times\mathbb C^t, \\
R_u(w_1,\ldots,w_s,z_{s+1},\ldots,z_{s+t}) &:= (w_1\cdot \sigma_1(u),\ldots,z_{s+t}\cdot \sigma_{s+t}(u)),
\end{split}
\end{gathered}
\end{equation*}
is properly discontinuous and co-compact.
One defines the {\em Oeljeklaus-Toma manifold} of type $(s,t)$ associated with the algebraic number field $K$ and to the admissible subgroup $U$ of $\mathcal O_K^{*,+}$ as
$$ X(K,U) :=\left. \mathbb H^s\times\mathbb C^t \middle\slash \mathcal O_K\rtimes U \right..$$
In particular, for an algebraic number field $K$ with $s=1$ real embedding and $2t=2$ complex embeddings, $X(K,\mathcal O_K^{*,+})$ is an Inoue-Bombieri surface of type $S_M$ \cite{inoue, bombieri}.

The Oeljeklaus-Toma manifold $X(K,U)$ is called {\em of simple type} if there exists no proper intermediate field extension $\mathbb{Q}\subset K^\prime \subset K$ with $U\subseteq\mathcal O_{K^\prime}^{*,+}$. This is equivalent to saying that there exists no holomorphic foliation of $X(K,U)$ with a leaf isomorphic to $X(K^\prime,U)$ \cite[Remark 1.7]{OT}.

\subsection{Solvmanifold structure} \label{sec:ot-solvmanifold}
Following \cite{k13}, we can endow Oeljeklaus-Toma manifolds with a structure of solvmanifold, (namely, a compact quotient of a solvable Lie group by a co-compact discrete subgroup,) in such a way that the complex structure is left-invariant, (namely, it descends from a complex structure on the Lie group that is invariant by left-translations). For later use, we recall here the construction of a left-invariant co-frame of $(1,0)$-forms and its structure equations, following the description in \cite[Section 6]{k13} and \cite[Section 2]{o20}.

On $\mathbb H^s \times \mathbb C^t$, take coordinates $(w^1,\ldots,w^s,z^1,\ldots,z^t)$.
By the very construction,
$$\left\{\left(\log\sigma_1(u), \ldots, \log\sigma_s(u)\right) : u\in U\right\}$$
yields a lattice in $\mathbb R^s$ of rank $s$.
In particular, fix any generators $\{ u_1, \ldots, u_s \}$ for $U$, then there exist real numbers $b_{k,i}\in\mathbb R$ and $c_{k,i}\in\mathbb R$, for $k\in\{1,\ldots,s\}$, $i\in\{1,\ldots,t\}$, such that, for any $i\in\{1,\ldots,t\}$,
$$
\left(
\begin{matrix}
b_{1,i} & \cdots & b_{s,i}
\end{matrix}
\right)
\cdot
\left(
\begin{matrix}
\log \sigma_1(u_1) & \cdots & \log\sigma_1(u_s) \\
\vdots & \ddots & \vdots \\
\log \sigma_s(u_1) & \cdots & \log\sigma_s(u_s)
\end{matrix}
\right)
=
\left(
\begin{matrix}
2\log|\sigma_{s+i}(u_1)| & \cdots & 2\log|\sigma_{s+i}(u_s)|
\end{matrix}
\right)
$$
and
$$
\left(
\begin{matrix}
c_{1,i} & \cdots & c_{s,i}
\end{matrix}
\right)
\cdot
\left(
\begin{matrix}
\log \sigma_1(u_1) & \cdots & \log\sigma_1(u_s) \\
\vdots & \ddots & \vdots \\
\log \sigma_s(u_1) & \cdots & \log\sigma_s(u_s)
\end{matrix}
\right)
=
\left(
\begin{matrix}
\arg\sigma_{s+i}(u_1) & \cdots & \arg\sigma_{s+i}(u_s)
\end{matrix}
\right) ,
$$
equivalently, for any $u\in U$ and any $i\in\{1,\ldots,t\}$, we have
$$ \sigma_{s+i}(u)=\prod_{k=1}^{s}\sigma_k(u)^{\frac{1}{2}b_{k,i}} \cdot \exp\left(\sqrt{-1}\sum_{k=1}^{s}c_{k,i}\log \sigma_k(u)\right).$$

The manifold $\mathbb H^s\times \mathbb C^t$ is endowed with the product
\begin{eqnarray*}
\lefteqn{ (\ldots, w^k,\ldots,z^i,\ldots) * (\ldots,\tilde w^k,\ldots, \tilde z^i,\ldots) } \\
&=&
\left( \ldots, \Re w^k+\Im w^k\cdot \Re \tilde w^k+\sqrt{-1}\Im w^k \cdot \Im\tilde w^k, \ldots, \right.\\
&& \left. \ldots, z^i+\prod_{k=1}^{s} (\Im w^k)^{\frac{1}{2}b_{k,i}}\exp\left(\sqrt{-1}\sum_{k=1}^{s}c_{k,i}\log (\Im w^k)\right)\tilde z^i, \ldots \right).
\end{eqnarray*}
It gives $\mathbb H^s\times \mathbb C^t$ a structure of solvable Lie group.
One can show that
$$ \mathcal O_K \rtimes U \subset \mathbb H^s \times \mathbb C^t, \qquad (u,a) \mapsto \left(\ldots, \sigma_k(a)+\sqrt{-1}\sigma_k(u),\ldots, \sigma_{s+i}(a),\ldots\right) $$
is a discrete subgroup of $(\mathbb H^s \times \mathbb C^t,*)$.

One can check that the following $(1,0)$-forms yield a left-invariant coframe for $(\mathbb H^s \times \mathbb C^t,*)$:
\begin{equation}\label{eq}
\left\{\begin{array}{ll}
\omega^k := \frac{1}{\Im w^k}dw^k, & \text{ for } k \in\{1,\ldots, s\}, \\
\gamma^i := \prod_{k=1}^{s} (\Im w^k)^{-\frac{1}{2}b_{k,i}}\exp\left(-\sqrt{-1}\sum_{k=1}^{s}c_{k,i}\log(\Im w^k)\right)dz^i, & \text{ for }i\in\{1,\ldots,t\},
\end{array}\right.
\end{equation}
with structure equations
\begin{equation}\label{eq:struct-eq}
\left\{\begin{array}{rcll}
d\omega^k&=&\frac{\sqrt{-1}}{2}\omega^k\wedge\bar\omega^k, & \text{ for } k \in\{1,\ldots, s\}, \\
d\gamma^i&=&\sum_{k=1}^{s}\left(\frac{\sqrt{-1}}{4}b_{k,i}-\frac{1}{2}c_{k,i}\right)\omega^k\wedge\gamma^i\\
&&+\sum_{k=1}^{s}\left(-\frac{\sqrt{-1}}{4}b_{k,i}+\frac{1}{2}c_{k,i}\right)\bar\omega^k\wedge\gamma^i & \text{ for }i\in\{1,\ldots,t\}.
\end{array}\right.
\end{equation}

\section{Metric properties of Oeljeklaus-Toma manifolds}\label{sec:metric}

On a complex manifold $X$ of complex dimension $n$ endowed with a Hermitian metric $g$, denote by $J$ the integrable almost-complex structure and by $\omega:=g(J\_,\_)$ the associated $(1,1)$-form. We recall some special classes of Hermitian metrics \cite{gray-hervella}:
\begin{itemize}
\item {\em K\"ahler} metrics satisfy $d\omega=0$;
\item {\em taming symplectic} (aka {\em Hermitian-symplectic}) structures  $\Omega$ are symplectic forms satisfying $\Omega(\_,J\_)>0$;
\item {\em pluriclosed} (aka {\em SKT}) metrics satisfy $\partial\overline\partial\omega=0$;
\item if $n\geq3$, {\em astheno-K\" ahler} metrics satisfy $\partial \overline{\partial}\omega^{n-2}=0$;
\item {\em balanced} metrics satisfy $d\omega^{n-1}=0$;
\item {\em Gauduchon} metrics satisfy $\partial\overline\partial\omega^{n-1}=0$;
\item {\em strongly Gauduchon} metrics satisfy $\partial\omega^{n-1}$ is $\overline\partial$-exact;
\item
for $k\in\{1,\ldots,n-1\}$, {\em $k$-Gauduchon} metrics satisfy $\partial\overline\partial\omega^k\wedge\omega^{n-k-1}=0$; notice that the case $k=n-1$ coincides with Gauduchon metrics, the case $k=1$ includes pluriclosed metrics, the case $k=n-2$ includes astheno-K\"ahler metrics;
\item {\em locally conformally K\"ahler} (shortly, {\em LCK}) metrics satisfy $d\omega=\theta\wedge\omega$ for some $1$-form $\theta$ such that $d\theta=0$;
\item {\em locally conformally balanced} (shortly, {\em LCB}) metrics satisfy $d\omega^{n-1}=\theta\wedge\omega^{n-1}$ for some $1$-form $\theta$ such that $d\theta=0$.
\end{itemize}

By \cite[Th\'eor\`eme 1]{gauduchon-cras1977}, Gauduchon metrics always exist on compact complex manifolds.
In what follows, we collect what is known about the existence of special Hermitian metrics on Oeljeklaus-Toma manifolds, after Oeljeklaus, Toma, Vuletescu, Dubickas, Otiman, Fino, Kasuya, Vezzoni:

\begin{theorem}[{\cite{OT, dub, fkv15, o20}}]\label{thm:ot-metrics}
Let $X(K, U)$ be an Oeljeklaus-Toma manifold of type $(s, t)$.
\begin{itemize}
\item It never admits K\"ahler metrics. \cite[Proposition 2.5]{OT}
\item It never admits taming symplectic forms. \cite[Theorem 1.1]{fkv15}
\item It admits pluriclosed metrics if and only if $s\leq t$ and, for any $u\in U$,
\begin{equation}\label{eq:ot-pluriclosed}
\begin{gathered}
\sigma_j(u)|\sigma_{s+j}(u)|^2=1 , \quad \text{ for any } j\in\{1,\ldots,s\} ,\\
|\sigma_{s+j}(u)|^2=1, \quad \text{ for any }j\in\{s+1,\ldots,t\} ,
\end{gathered}
\end{equation}
after possibly relabeling the embeddings. \cite[Theorem 1.1]{o20}
\item It never admits balanced metrics. \cite[Theorem 1.2]{o20}
\item It admits an LCK metric if and only if $|\sigma_{s+1}(u)|=\cdots=|\sigma_{s+t}(u)|$ for any $u\in U$. \cite[Proposition 2.9]{OT}, \cite[Theorem 8]{dub}
\begin{itemize}
\item If $s\geq1$ and $t=1$, then there is an admissible subgroup $U$ with this property. \cite[Proposition 2.9]{OT}
\item If $t\geq2$, then there is no admissible subgroup $U$ with this property, \cite{dv}.
\end{itemize}
\item It always admits an LCB metric. \cite[Theorem 1.3]{o20}
\end{itemize}
\end{theorem}

By methods similar to those used in \cite{dub} and \cite{dub2}, we shall prove:

\begin{theorem}\label{unitthm}
	Let $K$ be a number field with $s$ real embeddings and $2t$ complex 
	embeddings such that $$1 \leq s <t.$$
	Then, the smallest $d=s+2t$ for which there is a unit $u \in K$ of degree $d$ satisfying
		\begin{equation}\label{eq:ot-pluriclosed1}
			\sigma_j(u)|\sigma_{s+j}(u)|^2=1 
		\end{equation}
		for each  $j\in\{1,\ldots,s\}$ and  
		\begin{equation}\label{eq:ot-pluriclosed11}
			|\sigma_{s+j}(u)|=1 
		\end{equation}
		for each  $j \in \{s+1,\ldots,t\}$ is $d=12$.
		Moreover, there is no group of units 
		$U \subset\mathcal{O}_K^{*,+}$
		of rank $s$ such that for each $u \in U$ (not necessarily of degree $d$) we have  \eqref{eq:ot-pluriclosed1} and \eqref{eq:ot-pluriclosed11}.
\end{theorem}

\begin{proof}
	Assume that a unit $u \in K$ satisfies
	\eqref{eq:ot-pluriclosed11}. Then, $u$ must be a reciprocal unit.
	(An algebraic number $\alpha$ is called {\it reciprocal} if $\alpha^{-1}$ is a conjugate of $\alpha$ over $\mathbb Q$.)
	We claim that 
	\begin{equation}\label{prod1}
		\prod_{j=1}^s \sigma_j(u)=1
	\end{equation}
	for each reciprocal unit $u \in \mathcal{O}_{K}^{*}$ satisfying \eqref{eq:ot-pluriclosed1} and \eqref{eq:ot-pluriclosed11}. 
	
	Note that $\sigma_j(u)>0$ for $j=1,\ldots,s$ by \eqref{eq:ot-pluriclosed1}, and hence
	$u \in \mathcal{O}_{K}^{*,+}$.
	Clearly, \eqref{prod1} is true for $u=1$. Otherwise, $u \in \mathcal{O}_{K}^{*,+}$ is irrational.
	Since $u$ is reciprocal and $u \notin \mathbb Q$, it must be of even degree. Since ${\mathbb Q}(u)$ is a subfield of $K$, the degree $n=s+2t$ of $K$ must be even. Consequently, $s$ is even.

	Recall that $\sigma_j(u)$ are real and positive for all $j=1,\dots,s$.
	We claim that $\sigma_{s+j}(u)$ are non-real for all $j>0$. Indeed, if $\sigma_{s+j}(u)$ were real for some $j>0$ then
	$\overline\sigma_{s+j}(u)=\sigma_{s+j}(u)$. So, by \eqref{eq:ot-pluriclosed1}, we must have $$u'' u'^2=1$$ for some two conjugates $u''=\sigma_j(u),u'=\sigma_{s+j}(u)$ of $u$ over $\mathbb Q$. Take the smallest
	$\rho \geq 1$ such that all the conjugates of $u$ lie in the annulus $\rho^{-1} \leq |z| \leq \rho$, with at least one conjugate lying on the boundary. By Kronecker's theorem (see {\itshape e.g.} \cite[Theorem 4.5.4]{pra}), since $u$ is an algebraic integer which is not a root unity, we must have $\rho>1$.
		
	Consider an automorphism $\tau \in \{\sigma_1,\dots,\sigma_{s+2t}\}$ which maps $u'$ to a conjugate lying on the boundary of the annulus, say $u$. Then, 
	$$1=\tau(u''u'^2)= \tau(u'') \tau(u')^2= u^2 \tau(u''),$$ where $|u| \in \{\rho, \rho^{-1}\}$. However, if $|u|=\rho$ then $$1=|u|^2 |\tau(u'')|=\rho^2 |\tau(u'')| \geq \rho^2 \rho^{-1}=\rho,$$ which is impossible. Likewise, for $|u|=\rho^{-1}$ we obtain $1=\rho^{-2} |\tau(u')| \leq \rho^{-2}\rho=\rho^{-1}$, which is impossible too. Consequently, all
	$\sigma_{s+j}(u)$ for $j>s$ are non-real. 
	
	In particular, this implies that the list $\sigma_1(u), \dots, \sigma_{s}(u)$ contains all the real conjugates of $u$, possibly with repetitions. Suppose the reciprocal unit $u$ is of degree $2\ell$, where $\ell \in {\mathbb N}$, and its conjugates are  $$u_1,u_1^{-1},\dots,u_{\ell},u_{\ell}^{-1}.$$  Then $s$ must be a multiple of $2\ell$, and the product in \eqref{prod1} contains 
	$s/2\ell$ copies of each real conjugate of $u$. This clearly yields \eqref{prod1}. 
	
	Suppose now that $U$ is admissible for $K$, and for each $u \in U$ the conditions \eqref{eq:ot-pluriclosed1}, \eqref{eq:ot-pluriclosed11} are satisfied. Since $U$ is of rank $s$, it must contain
	$s$ units, say $v_1,\dots,v_s$ for which the matrix
	$\|\log \sigma_{j}(v_i)\|_{i,j=1,\dots,s}$ is nondegenerate. 
	However, as we have already shown above, each $v_i \in U$ is reciprocal, and hence
	$\sum_{j=1}^s \log \sigma_j(v_i)=0$ by \eqref{prod1}. 
	It follows that the rank of $U$ is less than $s$, which proves the second assertion of the theorem.
	
	In order to prove the first assertion we first show that no $K$ of degree less than $12$ exists for which a unit $u \in \mathcal{O}_K^{*,+}$
	satisfies \eqref{eq:ot-pluriclosed1} and \eqref{eq:ot-pluriclosed11}. Assume such a unit of degree $d=s+2t$ exists in the number field $K$ with signature $(s,t)$.
	Then, as $s$ is even and $s<t$, we have
	$d \geq 12$, unless $(s,t)=(2,3)$ or $(s,t)=(2,4)$. 
	
	We first investigate the case $(s,t)=(2,3)$. Suppose that the real conjugates of $u$ are $u_1>1$ and $u_2=u_1^{-1}<1$. 
	By \eqref{eq:ot-pluriclosed1}, it also has conjugates $u_3, u_4=\overline{u_3},u_5,u_6$, where $u_1 |u_3|^2=1$ and $\{u_5,u_6\}=\{u_3^{-1},u_4^{-1}\}$. Without loss of generality, we may assume that $u_5=u_3^{-1}$ and $u_6=u_4^{-1}=\overline{u_3}^{-1}=\overline{u_5}$. Also, by \eqref{eq:ot-pluriclosed11}, it has two conjugates $u_7$ and $u_8=\overline{u_7}=u_7^{-1}$ on the unit circle. 
	For convenience we present them in a table. 
	
	\begin{center}
		\begin{table}[h!]
			\begin{tabular}{| c | c || c | c | c | c | c | c |}
			
			\hline
			
			$u_1$ & $u_2$ & $u_3$ & $u_4$ & $u_5$ & $u_6$ & $u_7$ & $u_8$   \\ \hline
			
			& $u_1^{-1}$ &  & $\overline{u_3}$ & $u_3^{-1}$ & $\overline{u_5} = u_4^{-1}$ & & $\overline{u_7}=u_7^{-1}$   \\ \hline
		
				\end{tabular}
			\medskip
		\caption{Conjugates of $u$.}
		\label{table:metric22222}
	\end{table}
\end{center}

	Recall that $u_1>1$
	and $$u_1|u_3|^2=u_1u_3u_4=1.$$
	Take an automorphism
	$\sigma \in \{\sigma_1,\dots,\sigma_8\}$ which maps $u_1$ to $u_7$. Then, from $u_1=u_3^{-1}u_4^{-1}=u_5u_6$ we obtain
	$u_7=\sigma(u_5)\sigma(u_6)$. Since $|\sigma(u_5) \sigma(u_6)|=|u_7|=1$ and $u_7$ is non-real, the only possibilities for
	$\{\sigma(u_5),\sigma(u_6)\}$ are $\{u_4,u_5\}$ and $\{u_3,u_6\}$.
	Hence, $u_7=u_4u_5$ or $u_7=u_3u_6$. Since $u_2=u_3u_4$, multiplying it by $u_7=u_4u_5$ we get $u_2u_7=u_3u_4^2u_5=u_4^2$. This, mapping $u_4$ to $u_1$, gives a contradiction by the modulus consideration. Likewise, in the case when $u_7=u_3u_6$, we deduce $u_2u_7=u_3u_4u_3u_6=u_3^2$, which leads to the same contradiction by mapping $u_3$ to $u_1$. This shows that for $K$
	with signature $(2,3)$ no unit $u \in \mathcal{O}_K^{*,+}$ of degree $8$
	satisfying \eqref{eq:ot-pluriclosed1}, 
	\eqref{eq:ot-pluriclosed11} exists. 
	
	Now, we investigate the case $(s,t)=(2,4)$. Then, without loss of generality we can label the conjugates 
	$u_1,\dots,u_8$ as in Table~\ref{table:metric22222}. In addition, by 
	\eqref{eq:ot-pluriclosed11}, there are two more conjugates $u_9$ and $u_{10}=\overline{u_9}=u_9^{-1}$ on the unit circle.
	From $u_1=u_5u_6$, as above we obtain $u_7=\sigma(u_5)\sigma(u_6)$. This time,   the possibilities for	$\{\sigma(u_5),\sigma(u_6)\}$ are not only $\{u_4,u_5\}$, $\{u_3,u_6\}$, but also $\{u_8,u_9\}$ and $\{u_8,u_{10}\}$. 
	In the first two cases the argument is exactly the 
	same as that above. Assume that $\{\sigma(u_5),\sigma(u_6)\}=\{u_8,u_9\}$. 
	Then, from $u_7=u_8u_9$ we find that $u_7^2=u_7u_8u_9=u_9$.   Now, mapping $u_7$ to $u_1$, we arrive to a contradiction by the modulus consideration.
	(The case $\{\sigma(u_5),\sigma(u_6)\}=\{u_8,u_{10}\}$ can be treated in the same fashion.)
	Therefore, for $K$
	with signature $(2,4)$ no unit $u \in \mathcal{O}_K^{*,+}$ of degree $10$
	satisfying \eqref{eq:ot-pluriclosed1}, 
	\eqref{eq:ot-pluriclosed11} exists. 
	
	In order to complete the proof of the theorem we will give an example of a number field $K$ with signature $(2,5)$ and a positive unit $u \in K$ of degree $12$ satisfying \eqref{eq:ot-pluriclosed1} and \eqref{eq:ot-pluriclosed11}. 
	Select the numbers 
	$$ S:= -12+9 \sqrt{2}=0.727922\dots \>\> \text{and} \>\> P := 33-22\sqrt{2}=1.887301\dots.$$
	Let $R$ and $T$ be two complex conjugate numbers satisfying 
	$R+T=S$, $RT=P$ and $\Im(R)>0$. Consider the polynomial
	$$x^3+Rx^2-Tx-1$$ and its reciprocal $$x^3+Tx^2-Rx-1.$$
	Their product is a reciprocal polynomial
	\begin{align*}
		G(x)&:=x^6+Sx^5+(P-S)x^4+(2P-S^2-2)x^3+(P-S)x^2+Sx+1\\&=
		x^6-12x^5+45x^4-242x^3+45x^2-12x+1+ (9x^5-31x^4+172x^3-31x^2+9x)\sqrt{2}.
	\end{align*}
	We claim that it has six roots on the unit circle. 
	
	Indeed, setting $y=x+x^{-1}$, and using the identities $x^3+x^{-3}=y^3-3y$, $x^2+x^{-2}=y^2-2$ we derive that
	\begin{align*}
		x^{-3}G(x)&=y^3-3y+S(y^2-2)+(P-S)y+2P-S^2-2 \\&=y^3+Sy^2+(P-S-3)y+2P-S^2-2S-2 \\&=
		y^3+(-12+9\sqrt{2})y^2+(42-31\sqrt{2})y-218+154 \sqrt{2}.
	\end{align*}
	The latter polynomial (in $y$) has three roots $$y_1=-1.724350\ldots, \>\> y_2=-0.110593\ldots, \>\> y_3=1.107021\ldots$$ in the interval $(-2,2)$, so the six roots of $G$ coming from $x+x^{-1}=y_i$ are all on $|z|=1$. 
	
	Similarly, setting $$S':= -12-9 \sqrt{2}=-24.727922\ldots \>\> \text{and} \>\> P':= 33+22\sqrt{2}=64.112698\ldots,$$ we define
	two negative real numbers $R'= -21.784939\ldots$ and $T'=-2.942982\ldots$  which satisfy
	$R'+T'=S'$, $R'T'=P'$. The product of two reciprocal polynomials
	$$x^3+R'x^2-T'x-1 \>\> \text{and} \>\>x^3+T'x-R'x-1$$
	equals
	\begin{align*}
		\overline{G}(x)&:=x^6+S'x^5+(P'-S')x^4+(2P'-S'^2-2)x^3+(P'-S')x^2+S'x+1\\&=	x^6-12x^5+45x^4-242x^3+45x^2-12x+1- (9x^5-31x^4+172x^3-31x^2+9x)\sqrt{2}.
	\end{align*}
	This time, for $y=x+x^{-1}$ we find that
	$$x^{-3} \overline{G}(x)=	y^3+(-12-9\sqrt{2})y^2+(42+31\sqrt{2})y-218-154 \sqrt{2},$$
	and the latter polynomial (in $y$) has one real root 
	$21.697332\ldots$ and two complex conjugate roots. It follows that $\overline{G} \in {\mathbb R}[x]$ has two real roots, say $u_1=u=21.651145\ldots$ and $u_2=u^{-1}$, and two pairs of complex conjugate roots,
	$u_3, u_4=\overline{u_3}$ and $u_5=u_3^{-1}, u_6=\overline{u_5}=u_4^{-1}$. 
	Relabeling the conjugates if necessary, we can assume that
	$$x^3+R'x^2-T'x-1=(x-u_1)(x-u_3)(x-u_4)$$ and $$ x^3+T'x^2-R'x-1=(x-u_2)(x-u_5)(x-u_6).$$
	Here, we clearly have 
	\begin{equation}\label{lkop}
		u_1u_3u_4=1 \>\> \text{and} \>\> u_2u_5u_6=1.
	\end{equation}
	
	Suppose $u_7,\dots,u_{12}$
	are the six unimodular roots of $G$, so that
	\begin{equation}\label{lkop2}
		|u_7|=\dots=|u_{12}|=1. 
	\end{equation}
	Multiplying $G$ and $\overline{G}$ we find the reciprocal polynomial  
	$$ x^{12}-24x^{11}+72x^{10}-448x^9-191x^8-440x^7-432x^6-440x^5-191x^4-448x^3+72x^2-24x+1$$ in ${\mathbb Z}[x]$ with $12$ roots $u_1,\ldots,u_{12}$ satisfying \eqref{lkop} and \eqref{lkop2}. 
	One can easily verify with Maple that this polynomial is irreducible over $\mathbb Q$, so its root $u=21.651145\ldots$ is a reciprocal unit of degree $12$ satisfying the conditions \eqref{eq:ot-pluriclosed1} and \eqref{eq:ot-pluriclosed11}. 
	This completes the proof of the first claim of the theorem.
\end{proof}

As an immediate corollary, we get:

\begin{corollary}\label{rem:dubickas}
Let $X(K, U)$ be an Oeljeklaus-Toma manifold of type $(s, t)$.
It admits pluriclosed metrics if and only if $s = t$ and, for any $u\in U$,
\begin{equation}\label{eq:ot-pluriclosed-2}
\sigma_j(u)|\sigma_{s+j}(u)|^2=1 , \quad \text{ for any } j\in\{1,\ldots,s\} .
\end{equation}
\end{corollary}

For each $s \geq 1$ the construction of such fields $K$ with signature $(s,s)$ and such admissible 
group of units 
$U \subset\mathcal{O}_K^{*,+}$
of rank $s$ for which \eqref{eq:ot-pluriclosed-2} holds for any $u \in U$ was given in \cite{dub2}.

As for astheno-K\"ahler metrics, it is easy to prove that Oeljeklaus-Toma manifolds do not admit {\itshape left-invariant} Hermitian metric of this type.
We notice here that, as opposed to the balanced, pluriclosed, or taming symplectic cases, the averaging trick on solvmanifolds does not apply for astheno-K\"ahler metrics. Moreover, we will prove that astheno-K\"ahler metrics never exist on Oeljeklaus-Toma manifolds.

This will follow by considering a more general metric condition.
We recall that {\em $k$-Gauduchon metrics} \cite{fu-wang-wu} are defined by the property $\partial\overline\partial\omega^k\wedge\omega^{n-k-1}=0$. In what follows, we prove that $k$-Gauduchon metrics of ``special type'', including astheno-K\"ahler metrics, never exist on Oeljeklaus-Toma manifolds:

\begin{proposition}\label{prop:ddcomegak}
Let $X(K, U)$ be an Oeljeklaus-Toma manifold of type $(s, t)$ and complex dimension $n\geq4$.
Fix an arbitrary integer $k$ such that $2 \leq k \leq n-2$.
Then there is no Hermitian metric $\omega$ such that $\partial \overline{\partial}\omega^k=0$.
\end{proposition}

\begin{proof}
Consider the co-frame of left-invariant $(1,0)$-forms given by $\{\omega^1, \ldots, \omega^s, \gamma^1, \ldots, \gamma^t\}$, as presented in Section \ref{sec:ot-solvmanifold}.

\begin{description}
\item[Case $s \geq 2$ and $t \leq k \leq n-2$]
Consider the $(1, 1)$-form $\eta := - \sqrt{-1} \sum^{s}_{k, \ell=1}\omega^k \wedge \overline{\omega}^\ell$.
Following \eqref{eq:struct-eq}, we get $\partial\overline\partial\eta = \frac{\sqrt{-1}}{2} \sum_{k, \ell} \omega^k \wedge \bar{\omega}^k \wedge \omega^\ell \wedge \bar{\omega}^\ell$, which is a non-zero, strongly positive form of type $(2, 2)$.
For $I=\{i_1, \ldots ,i_{n-k-2}\}$, denote by $\omega_I := \omega_{i_1} \wedge \cdots \wedge \omega_{i_{n-k-2}}$.
Let $$\omega_0:=\sum_{|I|=n-k-2} \omega_I \wedge \overline{\omega}_I \wedge \eta.$$
Then $\partial\overline\partial \omega_0=\frac{\sqrt{-1}}{2} \sum_{|I|=n-k} \omega_I \wedge \overline{\omega}_I$ is a strongly positive $(n-k, n-k)$ form.
If $\omega$ such that $\partial\overline{\partial}\omega^k=0$ exists, we would get the following contradiction:
\begin{equation}\label{ineq}
0 < \int_X \partial\overline\partial \omega_0 \wedge \omega^k = \int_X \omega_0 \wedge \partial\overline\partial\omega^k=0.
\end{equation}
\item[Case $s=1$]
Let us fix $1 \leq k \leq n-2$. We shall prove there exists a positive $\sqrt{-1}\partial\overline\partial$-exact $(n-k, n-k)$-form. To this end, we claim there exist $1 \leq i_1 < \ldots < i_{n-k-1} \leq t$ such that 
$$\begin{gathered}
\rho := (b_{1,i_1}+\ldots + b_{1,i_{n-k-1}})(b_{1,i_1}+\ldots + b_{1,i_{n-k-1}}+1) \\ \cdot \sqrt{-1}^{n-k-1}\gamma_{i_1} \wedge \overline{\gamma}_{i_1} \wedge \ldots \gamma_{i_{n-k-1}} \wedge \overline{\gamma}_{i_{n-k-1}}\end{gathered}$$
satisfies $\sqrt{-1}\partial \overline{\partial}\rho \geq 0$ and it is not $0$. Using the structure equations given in \eqref{eq:struct-eq}, one can easily compute that:
$$\begin{gathered}
\sqrt{-1}\partial \overline{\partial} \rho= \frac{1}{4} (b_{1,i_1}+\ldots + b_{1,i_{n-k-1}})^2(b_{1,i_1}+\ldots + b_{1,i_{n-k-1}}+1)^2 \\
\cdot \sqrt{-1}^{n-k}\omega \wedge \overline{\omega} \wedge \bigwedge_{\ell=1}^{n-k-1} \gamma_{i_\ell} \wedge \overline{\gamma}_{i_\ell}.
\end{gathered}$$
 If there was no choice of $\{i_1, \ldots, i_{n-k-1}\}$ such that $\sqrt{-1}\partial \overline{\partial} \rho$ is not $0$, it would further imply that for any subset $\{i_1, \ldots, i_{n-k-1}\} \subset \{1, \ldots, t\}$, we have 
 \begin{equation}\label{sum}
 b_{1,i_1}+\ldots + b_{1,i_{n-k-1}} \in \{0, -1\}.
 \end{equation}
 Assume for now that this is the case. If $n-k-1 \in \{1, 2\}$, it is straightforward to see that \eqref{sum} combined with the fact that $\sum^t_{j=1}b_{1,j}=-1$ amounts to $b_{1,i_0}=-1$ for some $1 \leq i_0 \leq t$ and $b_{1,i}=0$ for $i \neq i_0$. If $t>1$, this further leads to $\sigma_1(u) \sigma_{1+i_0}(u) \overline{\sigma}_{1+i_0}(u) \equiv 1$ and $|\sigma_{1+i}(u)| \equiv 1$, for any $i \neq i_0$, which, however, is impossible, due to Theorem \ref{unitthm}.

 Let us suppose now that $n-k-1 \geq 3$. If there exist $i_m, i_p, i_\ell$ such that $b_{1,i_m} \neq b_{1,i_p}$, $b_{1,i_m} \neq b_{1,i_\ell}$ and $b_{1,i_\ell} \neq b_{1,i_p}$, then, since $n-k-1<t$, one can take $\{j_1, \ldots j_{n-k-3}\} \in \{1, \ldots, t\} \setminus \{i_m, i_\ell, i_p\}$ and therefore, the sums $$b_{1,i_m}+b_{1,i_\ell}+\sum^{n-k-3}_{t=1}b_{1,j_t}, \>\> b_{1,i_p}+b_{1,i_\ell}+\sum^{n-k-3}_{t=1}b_{1,j_t} \>\> \text{and} \>\> b_{1,i_m}+b_{1,i_p}+\sum^{n-k-3}_{t=1}b_{1,j_t}$$ are all different and cannot attain only the values $0$ and $-1$. Moreover, it is impossible to have $b_{1,1}= \ldots = b_{1,t}$, since this would contradict \eqref{sum} and $\sum^t_{j=1}b_{1,j}=-1$. Therefore, $b_{1,j}$, for any $1 \leq j \leq t$, can attain precisely two values. Let $m<M$ be these two values. Let $x$ denote the number of $b_{1,j}$ that are equal to $m$ and $y$ the number of $b_{1,j}$ that are equal to $M$. Then $x+y=t$ and $x \cdot m+y \cdot M=-1$. If $x < n-k-1$ and $y<n-k-1$, then we must have $2 \leq x$ and $2 \leq y$, since $n-k-1<t$. However, we cannot have the situation $x=y=2$, because this leads to the contradiction $2(m+M)=-1$, $m+2M=0$ and $2m+M=-1$. Therefore, in both cases $x \geq 3$ and $y \geq 3$, we can construct three different sums that cannot cover only the two values $0$ and $-1$. Finally, this gives that necessarily $x \geq n-k-1$ or $y \geq n-k-1$. Nevertheless, $x \geq n-k-1$ leads easily to a contradiction with \eqref{sum} and the case $y \geq n-k-1$ is the only one not conflicting with \eqref{sum} and giving $b_{1,i_0}=-1$ for precisely one $1 \leq i_0 \leq t$ and $b_{1,j}=0$ for $j \neq i_0$. If $t>1$, we further have again $\sigma_1(u) \sigma_{1+i_0}(u) \overline{\sigma}_{1+i_0}(u) \equiv 1$ and $|\sigma_{1+i}(u)| \equiv 1$, for any $i \neq i_0$, which is impossible by Theorem \ref{unitthm}. 
 
 The contradiction stems therefore from assumption \eqref{sum}, which means that there exists a choice of $\{i_1, \ldots, i_{n-k-1}\}$ such that $\sqrt{-1}\partial\overline{\partial} \rho$ is non-zero and positive. From here we argue like in \eqref{ineq} to deduce the inexistence of a Hermitian metric $\omega$ with $\sqrt{-1}\partial\overline{\partial} \omega^k=0$.
 
 \item[Case $s \geq 2$ and $2\leq k<t$]
Assume there exists a Hermitian metric $\omega$ such that $\partial \overline{\partial}\omega^k=0$. Define now 
\begin{equation*}
    \omega_0 (X_1, \ldots, X_{2k}):= \int_X \omega^k (X_1, \ldots, X_{2k}) d\mathrm{vol},
\end{equation*}
where $X_1, \ldots, X_{2k}$ are any left-invariant vector fields and $d\mathrm{vol}$ is a bi-invariant volume form. Then $\omega_0$ is a strictly positive left-invariant $(k, k)$-form that is $\partial\overline{\partial}$-closed. Note that $\omega_0$ is not necessarily the $k$-th power of a strictly positive $(1, 1)$-form.
We consider the frame $\mathcal{B}:=\left(\omega^1, \ldots, \omega^s, \gamma^1, \ldots, \gamma^t\right)$ and we denote by $\alpha^j$ its elements and by $\overline\beta^j$ their conjugates; we also denote by $\left(\omega_1, \ldots, \omega_s, \gamma_1, \ldots, \gamma_t\right)$ the dual basis. We have a decomposition of $\omega_0$ over the basis $\mathcal{B}$ as follows:
\begin{equation*}
    \omega_0 = \sum_{|I|=k, |J|=k} A (\alpha^I, \overline\beta^J)\alpha^{i_1} \wedge \overline\beta^{j_1} \wedge \ldots \wedge \alpha^{i_k} \wedge \overline\beta^{j_k},
\end{equation*}
where $A (\alpha^I, \overline{\beta}^J) \in \mathbb{C}$.
Since $\omega_0$ is strictly positive, we get that, for any $I:=\{i_1, \ldots i_k\} \subseteq \{1, \ldots t\}$, it holds $A(\gamma^I, \overline{\gamma}^I)>0$, where $\gamma^I:= \gamma_{i_1} \wedge \overline{\gamma}_{i_1}\wedge \cdots \wedge \gamma_{i_k} \wedge \overline{\gamma}_{i_k}$.
A straightforward computation using \eqref{eq:struct-eq} shows that, for any $I \subseteq \{1, \ldots, t\}$ and $h \neq \ell$,
\begin{align}\label{firsteq}
    0&=i_{\omega_\ell \wedge \overline{\omega}_h \wedge\gamma_I \wedge \overline{\gamma}_I} \partial \overline{\partial}\omega_0 = A(\gamma^I, \overline{\gamma}^I)i_{\omega_\ell \wedge \overline{\omega}_h \wedge\gamma_I \wedge \overline{\gamma}_I} \partial \overline{\partial} (\gamma^I \wedge \overline{\gamma}^I) \nonumber\\
    &= \frac{1}{4}A(\gamma^I, \overline{\gamma}^I)(b_{\ell,i_1}+\cdots+b_{\ell,i_k})(b_{h,i_1}+\cdots+b_{h,i_k})
\end{align}
and
\begin{align*}
    0&=i_{\omega_\ell \wedge \overline{\omega}_\ell \wedge\gamma_I \wedge \overline{\gamma}_I} \partial \overline{\partial}\omega_0 = A(\gamma^I, \overline{\gamma}^I)i_{\omega_\ell \wedge \overline{\omega}_\ell \wedge\gamma_I \wedge \overline{\gamma}_I} \partial \overline{\partial} (\gamma^I \wedge \overline{\gamma}^I)\\
    &= \frac{1}{2}A(\gamma^I, \overline{\gamma}^I)(b_{\ell,i_1}+\cdots+b_{\ell,i_k})(b_{\ell,i_1}+\cdots+b_{\ell,i_k}+1).
\end{align*}

This further implies that $b_{\ell,i_1}+\cdots+b_{\ell,i_k} \in \{0, -1\}$, for any $1 \leq \ell \leq s$ and for any $\{i_1, \ldots, i_k\} \subseteq \{1, \ldots, t\}$. Following now the same argument presented in the previous case, we obtain that for any $1 \leq \ell \leq s$, there exists $1 \leq j_\ell \leq t$ such that $b_{\ell,j_\ell}=-1$ and $b_{\ell,j}=0$, for any $j \neq j_\ell$. 
We take now $\ell \neq h$ and choose $\{i_1, \ldots, i_k\}$ such that $j_\ell, j_h \in \{i_1, \ldots, i_k\}$. This is possible since $2 \leq k$. Then we further have $(b_{\ell,i_1}+\cdots+b_{\ell,i_k})(b_{h,i_1}+\cdots+b_{h,i_k})=-1$, which is not possible by \eqref{firsteq}.
Therefore, there exists no metric $\omega$ such that $\partial \overline{\partial} \omega^k=0$, with $2 \leq k<t$.
\qedhere
\end{description}
\end{proof}

By Proposition \ref{prop:ddcomegak} for $k=n-2$, we immediately get the following

\begin{corollary}\label{cor:ak}
Oeljeklaus-Toma manifolds of complex dimension $n\geq4$ never admit astheno-K\"ahler metrics.
\end{corollary}

We further investigate the existence of {\it strongly Gauduchon metrics}, which are by definition Hermitian metrics $\omega$ such that $\partial \omega^{n-1}$ is $\overline{\partial}$-exact. They were introduced in \cite{pop13} and possess the remarkable property of openness under complex analytic deformation. One special subclass of such metrics are the balanced ones. We prove the following:

\begin{proposition}\label{prop:strongly-gauduchon}
Oeljeklaus-Toma manifolds $X(K, U)$ of any type do not admit strongly Gauduchon metrics.
\end{proposition}

\begin{proof}
Assume there exists a strongly Gauduchon metric $\omega$. Then, by definition, there exists an $(n, n-2)$-form $\alpha$ such that $\partial \omega^{n-1}=\overline{\partial} \alpha$. Let $\tau:= \frac{1}{2} \sum_{i=1}^s \overline{\omega}^i$, with the notation of \eqref{eq}. By Stokes, we have the following:
\begin{equation}\label{sG}
   \int_{X} \overline{\partial}\tau \wedge \alpha = \int_{X} \tau \wedge \overline{\partial} \alpha = \int_{X} \tau \wedge \partial \omega^{n-1} =  \int_{X} \partial \tau \wedge \omega^{n-1}.
\end{equation}

However, by \eqref{eq:struct-eq}, $\overline{\partial}\tau=0$ and $\partial \tau$ is a (strongly) positive $(1, 1)$-form. Then the first term of \eqref{sG} vanishes, but the last term is strictly positive and hence, there cannot exist any strongly Gauduchon metric on $X(K, U)$.
\end{proof}

\section{Cohomological properties of Oeljeklaus-Toma manifolds}\label{sec:cohom}

In what follows, we collect what is known about the cohomologies of Oeljeklaus-Toma manifolds, after Oeljeklaus, Toma, Tomassini, Torelli, Kasuya, Istrati, and the third-named author.
We use the following notation: for any multi-index $I=(1\leq i_1<\cdots<i_m\leq n)$ of length $|I|=m$, we consider $\sigma_I \colon U \to \mathbb C^{*}$ the representation $\sigma_I(u):=\sigma_{i_1}(u)\cdot\cdots\cdot\sigma_{i_m}(u)$. We denote by $\rho_m$ the number of non-trivial relations among the representations $\sigma_i$ with fixed length $m$, respectively, by $\rho_{p,m}$ the number of non-trivial relations among the $\sigma_i$ with fixed contribution from real and complex parts:
\begin{eqnarray}
\rho_m&:=&\sharp\left\{I : |I|=m, \sigma_{I}=1\right\} ,\label{eq:rho-dr}\\
\rho_{p,m}&:=&\sharp\left\{I\subseteq\{1,\ldots,s+t\}, J\subseteq\{s+1,\ldots,s+t\} : |I|=p, |J|=m, \sigma_{I}\bar\sigma_{J}=1\right\} . \label{eq:rho-dolb}
\end{eqnarray}

With this notation, we have the following:
\begin{theorem}[{\cite{OT, io19, k20, ot21}}]\label{thm:ot-cohom}
Let $X(K, U)$ be an Oeljeklaus-Toma manifold of type $(s, t)$.
\begin{itemize}
\item The Betti numbers are computed in \cite[Theorem 3.1]{io19}:
$$ b_k=\sum_{\ell+m=k} {s \choose \ell} \cdot \rho_m .$$
In particular:
\begin{itemize}
\item $b_1=s$ \cite[Proposition 2.3]{OT}
\item $b_2={s \choose 2}$ \cite[Proposition 2.3]{OT}, \cite[Proposition 2.4]{apv16}
\end{itemize}
\item The Hodge numbers are computed in \cite[Theorem 4.5]{ot21} and also \cite[Corollary 3.5]{k20}:
$$ h^{p,q}_{\overline\partial}=\sum_{\ell+m=q}{s \choose \ell} \rho_{p,m}. $$
In particular:
\begin{itemize}
\item $h_{\overline\partial}^{0,1}=s$ \cite[Corollary 4.6]{ot21}
\end{itemize}
\item The Hodge-Fr\"olicher spectral sequence degenerates at the first page. \cite[Theorem 4.5]{ot21}
\end{itemize}
\end{theorem}

\subsection{The $E_1$-isomorphism type and Bott-Chern cohomology}

Let $X:=X(K,U)$ be an Oeljeklaus-Toma manifold associated with the number field $K$ with $s$ real places and $t$ complex places and to the admissible subgroup $U$. Denote by $n=s+t$ its complex dimension. We order the embeddings $\sigma_i$ s.t. $\sigma_i=\bar\sigma_i$ for $i\in\{1,\dots,s\}$, and $\bar\sigma_j=\sigma_{t+j}$ for $j\in\{s+1,\dots,s+t\}$. With respect to the solvmanifold presentation of $X=G/\Gamma$ in Section \ref{sec:ot-solvmanifold}, denote the $G$-invariant $(1,0)$-forms by $\omega^k$, $\gamma^i$, for $k\in\{1,\ldots,s\}$, $i\in\{1,\ldots,t\}$.

We define a subalgebra of the algebra $A_X:=\mathcal{C}^\infty(X;\wedge^\bullet T^*X)$ of differential forms by
\[
B:=\Lambda\langle\omega^k,\bar\omega^k\rangle_{k\in\{1,...s\}}.
\]

Further consider the space
\begin{eqnarray}\label{eq:V}
V &:=& \operatorname{span}_{\mathbb{C}}\{dz^K \wedge d\bar{z}^L \mid K,L \subseteq\{1,\dots,t\}, \exists J \subseteq \{1,\ldots,s\}\text{ s.t. }\sigma_{JKL}|_U\equiv1\} \nonumber \\
&=& \operatorname{span}_{\mathbb{C}}\left\{ \exp(\Psi_{JKL}(x))\gamma^K\wedge\bar\gamma^L \mid J \subseteq \{1,\dots, s\},  K,L \subseteq\{1,\dots,t\} \text{ s.t. } \sigma_{JKL}|_U\equiv 1\right\},
\end{eqnarray}
where here we put $\sigma_{JKL}:=\prod_{j\in J}\sigma_j\prod_{k\in K}\sigma_{s+k}\prod_{\ell\in L}\sigma_{s+t+\ell}$ and
\begin{eqnarray*}
\Psi_{JKL}
&:=& -\sum_{j \in J}\log \Im w^j-\sum_{k \in K}\sum_{h=1}^s\left(\frac{1}{2}b_{h, k}-\sqrt{-1}c_{h, k}\right)\log \Im w^h\\
&&-\sum_{\ell \in L}\sum_{h=1}^s\left(\frac{1}{2}b_{h, \ell} +\sqrt{-1}c_{h, \ell}\right)\log \Im w^h \\
&=& -\sum_{j \in J}\log \Im w^j-\frac{1}{2}\sum_{k \in K}\sum_{h=1}^sb_{h, k}\log \Im w^h-\frac{1}{2}\sum_{\ell \in L}\sum_{h=1}^sb_{h, \ell}\log \Im w^h
\end{eqnarray*}
is as in \cite{k20}, and the simplification follows by $\sigma_{JKL}|_U\equiv1$. The space $V$ carries a grading by 
\[
V^r:=\operatorname{span}_{\mathbb{C}}\left\{ \exp(\Psi_{JKL}(x))\gamma_K\wedge\bar\gamma_L\in V\mid |J|+|K|=r\right\}\subseteq V.
\]

We denote by $VB$ the bigraded vector space with basis all wedges of elements in $V$ and $B$, {\itshape i.e.}
\begin{equation}\label{eq:VB}
VB:=V\wedge B\subseteq A_X.
\end{equation}
The grading of $V^r$ induces a grading on $VB$ via $V^rB:=V^r\wedge B$.

\begin{lemma}\label{lem: properties of VB}
	The subspace $VB\subseteq A_X$ is a bigraded, bidifferential subalgebra stable under conjugation. The grading of $V$ induces a direct sum decomposition $VB=\bigoplus_r V^rB$ into sub-double complexes which are also stable under conjugation.
\end{lemma}

\begin{proof}
	The statements that the space $VB$ is a bigraded subalgebra of $A_X$ and that $VB=\bigoplus V^rB$ as vector spaces are clear from the definitions. From the equations for the differentials, see \eqref{eq:struct-eq}, it follows that $VB$ and $V^rB$ are sub-double complexes. It only remains to show the statement about real structures, which follows from the relation $\sigma_{JKL}=\bar\sigma_{JLK}$.
\end{proof}

\begin{lemma}\label{lem: inclusion VB-AX quiso}
	The inclusion $\iota:VB\hookrightarrow A_X$ is an $E_1$-isomorphism.
\end{lemma}
\begin{proof}
The statement will follow from the following steps.

{\em Step 1: $VB$ is a direct summand in $A_X$.} We show the existence of a map of double complexes $r:A_X\to VB$ s.t. $r\circ \iota=\Id$. Namely, consider the left-invariant Hermitian metric  
	\[
	h:=\sum_{i=1}^s \omega_i\cdot\bar\omega_i+\sum_{i=1}^t \gamma_i\cdot\bar\gamma_i.
	\]
	We obtain an induced $L^2$-metric on $A_X$. The basis for $VB$, given by all elements of the type $\exp(\Psi_{JKL}(x))\omega_H\bar\omega_I\gamma_K\bar\gamma_L$ with $\sigma_{JKL}|_U\equiv 1$, is orthonormal with respect to this metric (this follows from the same calculations as in the proof of \cite[Lem. 2.3]{k20}). Now define $r$ to be the orthogonal projection to $VB$. Because $h$ is Hermitian, $r$ respects the bigrading and since $VB$ is closed under $d$ and $d^*$, one checks that $r$ is a chain map.
	
{\em Step 2: The inclusion is an $E_1$-isomorphism.} Since $\iota$ is compatible with the real structure, we only have to check the inclusion induces an isomorphism in Dolbeault cohomology. By {\em Step 1}, $H_{\delbar}(VB)\subseteq H_{\delbar}(X)$. On the other hand, the complex $(B_\Lambda^{\ast,\ast},\delbar)$ defined in \cite{k20} is a sub-complex of $(VB,\delbar)$ and it is shown in \cite{k20} that $H_{\delbar}(B_\Lambda)\cong H_{\delbar}(X)$, hence we obtain surjectivity.
\end{proof}

We now describe the $E_1$-isomorphism type of $A_X$ completely. This is equivalent to describing the multiplicities of all non-projective indecomposable bicomplexes ({\itshape i.e.} all `zigzags') in $A_X$. We use the notation in \cite{s21}, (in particular, an {\em odd length zigzag of shape $S^{p,q}_d$} has endpoints $(p, d-p)$ and $(d-q, q)$, length $2|p+q-d|+1$ and is concentrated in degrees $d$, $d - \mathrm{sgn}(p + q - d)$):

\begin{theorem}\label{thm: E1-isotype OT mfds}
	Let $X$ be an Oeljeklaus-Toma manifold. For any even length zigzag $Z$, one has
	\[
	\mult_{[Z]}(A_X)=0.
	\]
	For odd length zigzags of shape $S_d^{p,q}$, one has
	\[
	\operatorname{mult}_{S_d^{p,q}}(A_X)=\sum_{r}\operatorname{mult}_{S_d^{p,q}}(V^rB)
	\]
	and										
	\[
	\operatorname{mult}_{S_d^{p,q}}(V^rB)=
	\begin{cases} 
		h_{\bar\partial}^{p,d-p}(V^rB)&\text{ if } p=r, q=r\\
		0&\text{ otherwise.}
	\end{cases}
	\]
\end{theorem}

For example, if $X$ is a pluriclosed Oeljeklaus-Toma manifold with $s=t=2$, the only non-trivial relations are $\sigma_{1}|\sigma_3|^3=\sigma_2|\sigma_4|^2=1$ and one may picture the non-zero zigzags in $A_X$ as follows (we are only depicting those zigzags with non-zero multiplicity, but not writing the multiplicity):
\[
A_X\simeq_1 {V^0B}\oplus V^2B\oplus V^4B\simeq_1\img{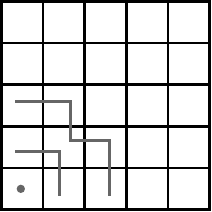}\oplus\img{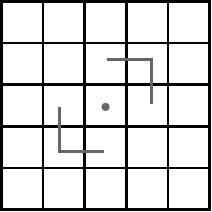}\oplus\img{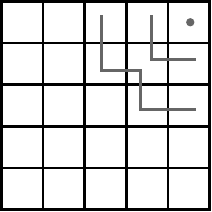}
\]
We will study the cohomologies of pluriclosed Oeljeklaus-Toma manifolds more in details in Section \ref{subsec:cohom-skt}.

We note that the Dolbeault cohomology of $X$ is known explicitly \cite{ot21, k20} and from this we have \cite[Theorem 3.3 and Corollary 3.5]{k20}:

\begin{lemma}\label{lem: Dolbeault dim OT}
	Let $X$ be an Oeljeklaus-Toma manifold of type $(s,t)$.
	Consider the number $\rho_{p,m}$ of non-trivial relations among the $\sigma_i$ with fixed contributions from real and complex parts as defined in \eqref{eq:rho-dolb}.
	Then,
	\[
	h_{\bar\partial}^{p,q}(V^rB)=\begin{cases} \sum_{q_1+q_2=q}{s\choose q_1}\rho_{p,q_2}&\text{ if }p=r\\
		0&\text{ otherwise.}
	\end{cases}
	\]
\end{lemma}

\begin{proof}[Proof of Theorem \ref{thm: E1-isotype OT mfds}]
	The statement about even length zigzags is equivalent to the degeneration of the Fr\"olicher spectral sequence of $X$ on the first page, which is known by \cite{ot21}. By Lemma \ref{lem: properties of VB} and Lemma \ref{lem: inclusion VB-AX quiso}, we have 
	\[
	\mult_{[Z]}(A_X)=\sum_{r}\mult_{[Z]}(V^rB)
	\] 
	for any zigzag $Z$. The multiplicities of odd length zigzags can be computed from the refined Betti-numbers, {\itshape i.e.} $\mult_{[S_d^{p,q}]}(V^rB)=b_d^{p,q}(V^rB):= \gr^p_F\gr^q_{\bar{F} }H^d_{dR}(V^rB)$, see \cite{s20,s21}, where $F$ denotes the Hodge filtration induced by the Fr\"olicher filtration $F^pA_X:=\bigoplus_{r\geq p}A^{r,\bullet}_X$ and $\bar F$ its conjugate. By Lemma \ref{lem: Dolbeault dim OT}, there is just one vertical strip of non-zero entries in $H^{\bullet,\bullet}(V^rB)$. Therefore, the Hodge filtration on $H_{dR}^d(V^rB)$ has exactly one breakpoint (if $b_d(V^rB)\neq 0$), namely,
	\[
	\{0\}=F^{r + 1}H_{dR}^d(V^rB)\subseteq F^{r} H_{dR}^d(V^rB)=H_{dR}^d(V^rB).
	\]
	As $V^rB$ is stable under conjugation, the same holds for the conjugate Hodge-filtration. Hence we get that for each $d$, there is at most one nonzero refined Betti-number $b_d^{p,q}(V^rB)$, {\itshape i.e.}:
	\[
	b_d^{p,q}(V^rB)=\operatorname{mult}_{S_d^{p,q}}(V^rB)=
	\begin{cases} 
		h_{\bar\partial}^{p,d-p}(V^rB)&\text{ if } p=r, q=r\\
		0&\text{ otherwise,}
	\end{cases}
	\]
completing the proof.
\end{proof}

From Theorem \ref{thm: E1-isotype OT mfds}, we may compute the dimensions of any cohomological invariant, by only computing how this invariant evaluates on odd length zigzags. We carry this out for Bott-Chern cohomology, but the reader will have no difficulty in doing the analogous procedure for Aeppli cohomology (which also follows by duality), the Varouchas groups, any higher-page analogues or the groups appearing in the Schweitzer complex.

\begin{corollary}\label{cor:bc-ot}
Let $X(K, U)$ be an Oeljeklaus-Toma manifold of type $(s, t)$.
	The Bott-Chern numbers of $X$ are given by:
	\[
	h_{BC}^{p,q}(X)=\sum_{r} h_{BC}^{p,q}(V^rB)
	\]
	and 
	\[
	h_{BC}^{p,q}(V^rB)=
	\begin{cases}
		h_{\bar\partial}^{r, p+q-r } &\text{ if } r\geq p,r\geq q,\\
		h_{\bar\partial}^{r, p+q-r-1} &\text{ if } r<p,r<q,\\
		0&\text{ otherwise.}
	\end{cases}
	\]
\end{corollary}

\begin{proof}
	We have $A_X\simeq_1\bigoplus_r V^rB$ and hence the first equality follows. For the second, by Theorem \ref{thm: E1-isotype OT mfds}, we have 
	\[
	h_{BC}(V^rB)=\sum_{d}h_{\delbar}^{r,d-r}(V^rB)h_{BC}^{p,q}([S_{d}^{r,r}]).
	\]
	Now, use the formulas
	\[
	h_{BC}^{a,b}([S_d^{p,q}])=\begin{cases}
		1 &\text{ if } a+b=d, a\leq p, b\leq q\\
		0&\text{ otherwise.}
	\end{cases}
	\]
	in the case $p+q\geq d$
	and 
	\[
	h_{BC}^{a,b}([S_d^{p,q}])=
	\begin{cases}
		1&\text{ if } a+b=d+1,~ a>p,b>q\\
		0&\text{ otherwise.}
	\end{cases}
	\]
	for $p+q< d$. These are obtained from \cite{s21} and easily seen directly as Bott-Chern cohomology counts top-right corners of zigzags. {\itshape E.g.}:
	\[
	H_{BC}\left(
	\begin{tikzcd}
		\langle a_1\rangle\ar[r]&\langle \del a_1\rangle\\
		&\langle a_2\rangle\ar[u]
	\end{tikzcd}\right)=
	\left(
	\begin{tikzcd}
		0&\langle\del a_1\rangle\\
		0&0
	\end{tikzcd}
	\right)
	\]
	and
	\[
	H_{BC}\left(
	\begin{tikzcd}
		\langle \delbar a\rangle&\\
		\langle a\rangle\ar[r]\ar[u]&\langle \del a\rangle
	\end{tikzcd}\right)=
	\left(	
	\begin{tikzcd}
		\langle\delbar a\rangle&0\\
		0&\langle\del a\rangle
	\end{tikzcd}
	\right)
	\]
\end{proof}

\begin{remark} One can in principle also argue via explicit representatives, {\itshape i.e.} explicitly compute the Bott-Chern cohomology of $V^\bullet B$. However, such representatives are not necessarily of a simple form; in particular, they are generally not elementary wedges, but rather sums of such. For example, if $s=t=2$, the natural map
	\[
	H_{BC}^{2,2}(X)\to H_{\bar\partial}^{2,2}(X)
	\]
	is an isomorphism, but the preimage of $[\gamma_{1\bar 1}\omega_{1\bar 2}]$ cannot be represented by the same form as it is not $\partial$-closed. Rather, one has to modify it by \[
	\bar\partial(\gamma_{1\bar 1}\omega_2)=\frac{\sqrt{-1}}{2}(\omega_{\bar 2 1}\gamma_{1\bar 1} - \gamma_{1\bar 1}\omega_{2\bar 2})
	\]
	to obtain something in the kernel of $\partial $ and $\bar\partial$.
\end{remark}

\subsection{Cohomological properties of pluriclosed Oeljeklaus-Toma manifolds}\label{subsec:cohom-skt}
In this section, we restrict to Oeljeklaus-Toma manifolds admitting pluriclosed metrics, as characterized in Corollary \ref{rem:dubickas}, and we study their cohomological properties. In particular, we prove that the Dolbeault cohomology is invariant (see Remark \ref{rem:inv-pluriclosed}) and we explicitly compute the Betti and Hodge numbers (see Theorem \ref{cohomology}): note that the cohomologies are not zero most of the times.

Recall that $\rho_m$, $\rho_{p,m}$ were introduced in \eqref{eq:rho-dr}, respectively \eqref{eq:rho-dolb}. It is clear that $\rho_0=1$ and $\rho_1=0$ always; one can prove that $\rho_2=0$ always \cite[Proposition 2.4]{apv16}. When $X(K,U)$ admits pluriclosed metrics, then it is proven that $\rho_3=s$ and $\rho_{2,1}=s$, therefore $b_3={s\choose 3}+s$ and $h^{2,1}_{\overline\partial}=s$ \cite[Proposition 3.2.1]{o20}. Moreover, for $n=4$, namely, $s=t=2$, the converse holds true: $X$ admits a pluriclosed metric if and only if $b_3=2$ if and only if $h^{2,1}_{\overline\partial}=2$ \cite[Proposition 3.2.2]{o20}.
In this section, we generalize the above results.

Meanwhile, we notice the following:
\begin{proposition}
Oeljeklaus-Toma manifolds admitting pluriclosed metrics are of simple type.
\end{proposition}

\begin{proof}
Assume there exists a finite extension $\mathbb{Q} \subseteq K' \subseteq K$ such that $U \subseteq K'$; we take $K'=\mathbb Q(U)$. The embeddings of $K'$ are simply the restrictions of the embeddings of $K$, namely $\sigma_1|_{K'}, \ldots, \overline{\sigma}_{2s}|_{K'}$. Since $U$ is admissible and $U \subseteq K'$, $\sigma_i|_{K'} \neq \sigma_j|_{K'}$, for any $1 \leq i<j \leq s$. However, as the pluriclosed condition is satisfied, $\sigma_i|_{U} \sigma_{s+i}|_{U} \overline{\sigma}_{s+i}|_{U} \equiv 1$ and therefore, $\sigma_i|_{K'} \sigma_{s+i}|_{K'} \overline{\sigma}_{s+i}|_{K'} \equiv 1$, which further implies that $\sigma_{s+i}|_{K'} \neq \sigma_{s+j}|_{K'}$, for any $1 \leq i<j \leq s$. Moreover, $\sigma_{s+i}|_{K'} \neq \sigma_{j}|_{K'}$, for any $1 \leq i,j \leq s$, otherwise this would imply a non-trivial multiplicative relations between $\sigma_1|_{K'}, \ldots, \sigma_s|_{K'}$, which is impossible. Hence, the degree $[K': \mathbb{Q}] \geq 2s$. As $[K: \mathbb{Q}]=[K: K'] \cdot [K': \mathbb{Q}]$ and $[K: \mathbb{Q}]=3s$, we infer that $K=K'$ and conclude that $X(K, U)$ is of simple type.
\end{proof}

We now compute the de Rham and Dolbeault cohomologies of pluriclosed Oeljeklaus-Toma manifolds in the following:
\begin{theorem}\label{cohomology}
Let $X(K, U)$ be an Oeljeklaus-Toma manifold admitting a pluriclosed metric. Then the Betti and Hodge numbers are:
$$b_\ell=\sum_{k \leq s} {s \choose \ell-3k} \cdot {s \choose k} ,
\quad\quad
h^{p, q}_{\overline\partial}= \left \{
    \begin{array}{ll}
      {s \choose q-\frac{p}{2}} \cdot {s \choose \frac{p}{2}},  & \mbox{if } p \, \,  \mbox{even}\\
        0, & \mbox{otherwise.}
    \end{array} \right. $$
\end{theorem}

\begin{proof}
{\em Step 1:} By Corollary \ref{rem:dubickas}, we have $s=t$ and $|\sigma_{s+i}(u)|^2\sigma_i(u)=1$, for any $1 \leq i \leq s$, for any $u\in U$. 

By \cite[Theorem 3.1]{io19}:
\begin{equation}\label{betti}
    b_\ell= \rho_\ell+ {s \choose 1} \rho_{\ell-1} + {s \choose 2} \rho_{\ell-2}+ \cdots + {s \choose \ell-2} \rho_{2} + {s \choose \ell},
\end{equation}
where $\rho_m$ is defined in \eqref{eq:rho-dr}. 
We prove now that $\rho_{3k}={s \choose k}$ and $\rho_m=0$, if $m \neq 3k$. Indeed, let $I=\{i_1, \ldots, i_m\}$ be such that $\sigma_I(u)=1, \forall u \in U$. Then 
\begin{equation}\label{conjugate}
    \sigma_I(u) \cdot \overline{\sigma_I(u)}=1, \forall u \in U.
\end{equation}

If $I \neq \cup_{j \in J} \{j, s+j, 2s+j\}$, where $J \subseteq \{1, \ldots, s\}$, then, since for any $i_r \geq s+1$, $\sigma_{i_r}(u) \cdot \overline{\sigma_{i_r}(u)}=\frac{1}{\sigma_{i_r-s}(u)}$ thanks to the pluriclosed condition, we get that \eqref{conjugate} would further imply a multiplicative relation between $\sigma_1|_U, \ldots, \sigma_s|_U$, which is impossible by the admissibility of $U$.

Therefore, if $\sigma_I|_U \equiv 1$, then $I = \cup_{j \in J} \{j, s+j, 2s+j\}$ and consequently, $\rho_m=0$, if $m \neq 3k$ and
\begin{equation*}
    \rho_{3k}=\sharp\{J \mid |J|=k, J \subseteq \{1, \ldots, s\}\}={s \choose k}.
\end{equation*}
Applying \eqref{betti}, we get the formula for the Betti numbers.

{\em Step 2:} By \cite[Relation (4.11)]{ot21} we have the following formula for the Hodge numbers:
\begin{equation}\label{hodge}
    h^{p, q}_{\overline{\partial}}=\sum_{\ell+m=q} {s \choose \ell} \rho_{p,m} ,
\end{equation}
where $\rho_{p,m}$ is defined in \eqref{eq:rho-dolb}.

By the previous characterization of non-trivial multiplicative relations between the embeddings, we obtain that $\sigma_{I \cup J}|_U \equiv 1$ (where $I$ and $J$ are as in \eqref{hodge}) if and only if $p=2j$ and $I \cup J= \cup_{k \in K} \{k, s+k, 2s+k\}$, where $K \subseteq \{1, \ldots, s\}$. This further means that:
\begin{equation*}
    \sharp\{ I \subseteq \{1, \ldots s+t\}, J \subseteq \{s+t+1, \ldots s+2t\} \mid |I|=p, |J|=j, \sigma_{I \cup J}|_U \equiv 1\}={s \choose \frac{p}{2}}
\end{equation*}
and we obtain the formulas for the Hodge numbers for pluriclosed $X(K, U)$.
\end{proof}

\begin{remark}\label{rem:inv-pluriclosed}
Notice that pluriclosed Oeljeklaus-Toma manifolds $X(K,U)$ have invariant de Rham and Dolbeault cohomology, in the sense that $H^\bullet(X(K, U)) \simeq H^\bullet(\mathfrak{g})$ and $H_{\overline{\partial}}^{\bullet,\bullet}(X(K, U)) \simeq H_{\overline{\partial}}^{\bullet,\bullet}(\mathfrak{g})$, where $\mathfrak{g}$ is the Lie algebra of the Lie group $(\mathbb H^s\times\mathbb C^t,*)$. Indeed, this is an immediate consequence of the fact that 
$$H^{\bullet, \bullet}_{\overline{\partial}}(X(K, U))= \left\langle \frac{d\overline{w}_1}{\Im w_1}, \ldots, \frac{d\overline{w}_s}{\Im w_s}, dw_1 \wedge dz_1 \wedge d\overline{z}_1, \ldots,  dw_s \wedge dz_s \wedge d\overline{z}_s  \right\rangle_{\mathbb{C}}$$
and all the listed forms are left-invariant with respect to the solvmanifold structure described in Section \ref{sec:ot-solvmanifold}.
\end{remark}

Generalizing \cite[Proposition 3.2.2]{o20}, we have the following cohomological characterization of the pluriclosed condition:
\begin{theorem}\label{thm:cohom-skt}
Let $X(K, U)$ be an Oeljeklaus-Toma manifold of type $(s,t)$. The following statements are equivalent:
\begin{enumerate}
    \item $X(K, U)$ admits a pluriclosed metric;
    \item $s=t$, $h_{\overline\partial}^{2, 1}=s$, $h_{\overline\partial}^{4, 2}={s \choose 2}$, $h_{\overline\partial}^{1, 2}=0$ and the natural map
    \begin{equation*}
       \bigwedge: H_{\overline{\partial}}^{2, 1}(X) \times H_{\overline{\partial}}^{2, 1}(X) \rightarrow H_{\overline{\partial}}^{4, 2}(X)
    \end{equation*}
    is surjective.
    \end{enumerate}
\end{theorem}

\begin{proof}
The implication $(1) \Rightarrow (2)$ is clear by Theorem \ref{cohomology} and the fact that $H_{\overline{\partial}}^{2,1}(X)$ is generated by $\{[d w_i \wedge dz_i \wedge d\overline{z}_i]\}_{i}$ and $H_{\overline{\partial}}^{4,2}(X)$ by $\{[d w_i \wedge d w_j \wedge dz_i \wedge dz_j \wedge d\overline{z}_i \wedge d\overline{z}_j]\}_{ij}$, see Remark \ref{rem:inv-pluriclosed}.

We prove now $(2) \Rightarrow (1)$. By \eqref{hodge} and $h_{\overline\partial}^{2, 1}=s$, we have that there are precisely $s$ non-trivial multiplicative relations between the embeddings:
$$
\begin{gathered}
\sigma_{i_{11}} (u) \cdot \sigma_{i_{12}} (u) \cdot \sigma_{i_{13}} (u)=1, \forall u \in U,\\
\ldots\\
\sigma_{i_{s1}} (u) \cdot \sigma_{i_{s2}} (u) \cdot \sigma_{i_{s3}} (u)=1, \forall u \in U,
\end{gathered}
$$
where $i_{l1} < i_{l2} < i_{l3}$, for any $1 \leq l \leq s$. This further implies that $H_{\overline{\partial}}^{2, 1}(X)$ is generated by $\{[dz_{i_{l1}} \wedge dz_{i_{l2}} \wedge d\overline{z}_{i_{l3}}]\}_{1 \leq l \leq s}$. Take now 
$$\bigwedge: H_{\overline{\partial}}^{2, 1}(X) \times H_{\overline{\partial}}^{2, 1}(X) \rightarrow H_{\overline{\partial}}^{4, 2}(X)$$
given by $[\alpha] \wedge [\beta] \longmapsto [\alpha \wedge \beta]$. If for some $k \neq l$ we had $\{i_{k1}, i_{k2}, i_{k3}\} \cap \{i_{l1}, i_{l2}, i_{l3}\} \neq \emptyset$, then $[dz_{i_{k1}} \wedge dz_{i_{k2}} \wedge dz_{i_{k3}}] \wedge [dz_{i_{l1}} \wedge dz_{i_{l2}} \wedge dz_{i_{l3}}]=0$ and therefore, $\mathrm{dim}_{\mathbb{C}} \Im \bigwedge < {s \choose 2}$. However, $\bigwedge$ is surjective, hence $\{i_{k1}, i_{k2}, i_{k3}\} \cap \{i_{l1}, i_{l2}, i_{l3}\} = \emptyset$ for different $k$ and $l$ and 
$$\bigcup_{1 \leq k \leq s} \{i_{k1}, i_{k2}, i_{k3}\}=\{1, \ldots, s+2t\}.$$
Moreover, each of the $s$ multiplicative relations contains exactly one real embedding ({\itshape i.e.} $i_{k1} \leq s$, for all $1 \leq k \leq s$). Indeed, if this was not true, then a certain multiplicative relation $\sigma_{i_{k1}} \cdot \sigma_{i_{k2}} \cdot \sigma_{i_{k3}}|_{U}\equiv 1$, would have $s< i_{k1}< i_{k2}$. Since in this case the conjugate relation $\overline{\sigma}_{i_{k1}} \cdot \overline{\sigma}_{i_{k2}} \cdot \overline{\sigma}_{i_{k3}}|_{U}\equiv 1$ would produce a non-zero element in $H_{\overline{\partial}}^{1, 2}(X)$, which is impossible, we get our conclusion. Now, since $i_{k1} \leq s$, for every $1 \leq k \leq s$, we get that necessarily $\sigma_{i_{k2}}$ and $\sigma_{i_{k3}}$ are conjugated to each other, otherwise $\overline{\sigma}_{i_{k1}} \cdot \overline{\sigma}_{i_{k2}} \cdot \overline{\sigma}_{i_{k3}}|_{U}\equiv 1$ would be a different relation, still involving the term $\sigma_{i_{k1}}$, which is again impossible. Therefore the $s$ multiplicative relations are exactly the numerical condition described in \cite[Theorem 3.2]{o20}, see Corollary \ref{rem:dubickas}, which is equivalent to the existence of a pluriclosed metric. 
\end{proof}

\end{document}